\documentclass{svproc}
\usepackage{url}

\usepackage{amsmath,amssymb,color}

\newcommand{\R}{\mathbb{R}}

\newcommand{\C}{\mathbb{C}}

\newcommand{\N}{\mathbb{N}}

\newcommand{\eps}{\varepsilon}

\DeclareMathOperator{\Rp}{Re\,}

\begin{document}
\mainmatter              
\title{Fourth-order operators with unbounded coefficients in $L^1$ spaces}
\titlerunning{Fourth-order operators with unbounded coefficients in $L^1(\R^N)$}  
%
\author{Federica Gregorio\inst{1} \and Chiara Spina\inst{2}
  \and Cristian Tacelli\inst{1}}
\authorrunning{F. Gregorio, C. Spina, C. Tacelli} 
\institute{Dipartimento di Matematica, Università degli Studi di Salerno, Fisciano, Italy\\
\email{fgregorio@unisa.it},\\\email{ctacelli@unisa.it}
\and
Dipartimento di Matematica e Fisica "Ennio De Giorgi", Università del Salento, Lecce, Italy\\\email{chiara.spina@unisalento.it}}

\maketitle              

\begin{abstract}
We prove that  operators of the form $A=-a(x)^2\Delta^{2}$, with  suitable growth conditions on the coefficient $a(x)$,  generate  analytic semigroups in $L^1(\R^N)$. In particular, we deduce generation results for the operator $A :=- (1+|x|^2)^{\alpha} \Delta^{2}$, $0\leq\alpha\leq2$.  Moreover, we characterise the maximal domain of $A$ in $L^1(\R^N)$. 
\end{abstract}

\section{Introduction}

In this paper we are interested in elliptic and parabolic solvability in $L^1(\R^N)$   of the problems associated with  fourth order operators
of the form $$A:=-a(x)^2\Delta^{2}$$ with suitable growth conditions on the coefficient $a(x)$. In particular, the results proved in the paper apply to the model operator $A =- (1+|x|^2)^{\alpha} \Delta^{2}$
in the case $0\leq\alpha\leq2$.
In the language of semigroups, we prove that $A$, endowed  with the domain
$$D(A_1)=\{u\in L^1(\R^N):\ a^\frac{1}{2}D u,\ a D^2u,\ a^\frac{3}{2}D^3 u,\ a^2\Delta^2 u\in L^1(\R^N)\},$$
generates an analytic semigroup.

Although second order elliptic operators with unbounded coefficients and, in particular, polynomially growing coefficients, have been widely studied, see for example \cite{for-lor,met-spi2,met-oka-sob-spi,bcgt}, where unbounded coefficients similar to ours have been introduced,  it seems that the case of higher order operators has been less addressed.
 
The semigroup generated by a class of higher order elliptic operators with  bounded measurable coefficients in $L^p(\R^N)$ has been systematically studied by Davies, cf. \cite{dav95a}. Unbounded coefficients for higher order operators have been considered in recent times. In \cite{AGRT} the authors obtained generation results for the square of the Kolmogorov operator $L=\Delta+\frac{\nabla\mu}{\mu}\cdot\nabla$ in the weighted space $L^2(\R^N,d\mu)$. 
In \cite{gre-tac} polynomially growing coefficients are considered. More precisely, the operator $A=(1+|x|^\alpha)^2\Delta^2+|x|^{2\beta}$ with $\alpha>0$ and $\beta>(\alpha-2)^+$ has been studied in the space $L^2(\R^N)$ for $N>4$.

The present paper extends the results proved in  \cite{gre-spi-tac1} where we studied the operator $A$ in $C_b(\R^N)$ and $L^p(\R^N)$ spaces for $1<p\leq\infty$. By proving the sectoriality of the operator $A$, we obtained generation of an analytic semigroup   and characterised the maximal domain of such operators in $L^p(\R^N)$ for $1<p<\infty$.

Here we prove generation results also in $L^1(\R^N)$ by proving that the operator $(A,D(A_1))$ is sectorial. 
To get the main result, we first prove the resolvent estimate in $L^1$   in a suitable sector. In order to do this we use a duality argument which is based upon the  generation result in continuous function spaces for the adjoint operator 
$$-\hat Au =a^2\Delta ^2u+4D a^2D \Delta u+2\Delta a^2\Delta u+4tr D^2a^2D^2u+4D \Delta a^2D u+\Delta ^2a^2 u.$$
Since the latter operator contains also explicit lower order terms with unbounded coefficients, we need a generalization of the above mentioned results in $L^p(\R^N)$, $1<p\leq\infty$, and then in $C_b(\R^N)$ for operators of this form as a preliminary result. For this reason  Section \ref{lower} is devoted to the study of  more general operators with lower order terms via perturbation methods.  Once the resolvent estimate is known, an approximation argument with operators with smooth and bounded coefficients allows us to prove that a suitable sector is contained in the resolvent set.

We have a precise description of the domain, indeed we prove that the maximal domain of $A$ in $L^1(\R^N)$ coincides with $D(A_1)$.

\medskip

 \textbf{Notations.} We use standard notations for function spaces. We denote by $L^p(\R^N)$ and $W^{k,p}(\R^N)$ the standard Lebesgue and Sobolev spaces, respectively. $C_c^\infty(\R^N)$ is the space of test functions and $C_b(\R^N)$ is the space of bounded continuous functions. We denote by $W_{\rm loc}^{3,p}(\R^{N})$   the space of functions in $W^{3,p}(\Omega)$ for every compact subset $\Omega$ of $\R^N$.
$B_R$ denotes the open ball of radius $R$ centred at 0, whereas $B_R(x_0)$ is the ball centred at $x_0$ for some $x_0\in\R^N$.  We denote by $B_R^c$ the complementary set of the open ball, $B_R^c=\R^N\setminus B_R$. Moreover, we use the symbol $\|\cdot\|_{p,R}$ to shorten the notation $\|\cdot\|_{L^p(B_R(x_0))}$. In the proofs, $C$ is a  positive constant that may vary from line to line. Finally, for every function $u$ smooth enough we set
\begin{align*}|D^4u|=\left(\sum_{i,j,k,l=1}^N|D_{ijkl}u|^2\right)^\frac12&, \ |D^3u|=\left(\sum_{i,j,k=1}^N|D_{ijk}u|^2\right)^\frac12\\|D^2u|=\left(\sum_{i,j=1}^N|D_{ij}u|^2\right)^\frac12&, \ |Du|=\left(\sum_{i=1}^N|D_{i}u|^2\right)^\frac12.\end{align*}
We say that an operator $A$ is sectorial operator if there exist $\omega\in\R,\ \theta\in\left(\frac\pi2,\pi\right),C>0$ such that $\rho(A)\supset\Sigma=\{\lambda\in\C:\lambda\neq\omega,|\textrm{arg}(\lambda-\omega)|<\theta\}$ and the resolvent estimate $\|R(\lambda,A)\|\leq\frac{C}{|\lambda-\omega|}$ holds for every $\lambda\in\Sigma$.

\section{Generation results  for more general operators}\label{lower}
In this section we study generation results in $C_b $ and in $L^p$ spaces, $1<p\leq\infty$, for  operators with lower order terms of the form
\[
-Lu=a^2\Delta^2u+\sum_{i,j,k=1}^Nb_{ijk}D_{ijk}u+\sum_{i,j=1}^Nc_{ij}D_{ij}u+\sum_{i=1}^Nd_iD_iu+eu
\]
where the lower order coefficients satisfy the following growth assumptions 
\begin{alignat*}{2} \label{gradient-p}
0<a \in C^1(\R^N),\quad
 &|D a(x)|\leq \kappa a(x)^\frac{1}{2},\\
 b_{ijk}\in C(\R^N),\quad &|b_{ijk}(x)|\leq \kappa  a^\frac{3}{2}(x) \quad &&\forall\,i,j,k=1,\dots, N,\\
\tag{$H_p$}  c_{ij}\in C(\R^N),\quad &|c_{ij}(x)|\leq \kappa^2 a(x) &&\forall\,i,j=1,\dots, N,\\
  d_{i}\in C(\R^N),\quad &|d_{i}(x)|\leq \kappa^3 a^\frac{1}{2}(x) &&\forall\,i=1,\dots, N,\\
  e\in C(\R^N),\quad &|e(x)|\leq \kappa ^4,
\end{alignat*}
for all $x\in \R^N$ and some positive constant $\kappa$. We endow $L$  with the following domain
$$D(A_p)=\{u\in W_{\rm loc}^{4,p}(\R^N)\cap L^p(\R^N):\ a^\frac{1}{2}D u,\ a D^2u,\ a^\frac{3}{2}D^3 u,\ a^2D^4 u\in L^p(\R^N)\}.$$

We write the operator $L$ as a sum operator $$L=A+B$$ where $A=-a^2\Delta^2$ is the principal part of $L$. The operator $A$ has been studied in \cite{gre-spi-tac1}. In particular, the following generation result holds for $(A,D(A_p))$ in $L^p(\R^N), 1<p<\infty$. 

\begin{theorem}
Let  $1<p<\infty$ and assume that $a$ satisfies (\ref{gradient-p}). Then, $(A,D(A_p))$ generates an analytic semigroup in $L^p(\R^N)$.
\end{theorem}

The generation of analytic semigroups was showed by proving the sectoriality of $(A,D(A_p))$ in a  sector  depending only on $N,p$ and $\kappa$.

Having in mind to extend the result to the operator $L$, we recall   the following density result  (see \cite[Lemma 2.2]{gre-spi-tac1} for the proof). 
\begin{lemma} \label{density}
Let  $1<p<\infty$ and assume that $a$ satisfies (\ref{gradient-p}). Then $C_c^\infty(\R^N)$ is dense in $D(A_p)$ with respect to the norm
$$\|u\|_{D(A_p)}=\|u\|_p+\|a^\frac{1}{2}D u\|_p+\|a D^2u\|_p+\|a^\frac{3}{2}D^3 u\|_p+\|a^2D^4 u\|_p.$$
\end{lemma}
By arguing as in \cite[Lemma 2.4]{gre-spi-tac1} we prove the following interpolation estimates. The main novelty is in the last inequality 
(\ref{eq:calderon-zigmund}) where the right hand side contains the whole operator $L$. 

\begin{lemma}  \label{interpolation}
Let  $1<p<\infty$ and assume that $a,b,c,d,e$ satisfy (\ref{gradient-p}). Then, there exists a positive constant  $C=C(N,p,\kappa)$ such that for every $\varepsilon>0$ the following inequalities hold
\begin{align}
\|a^{\frac{1}{2}}Du\|_{p}&\leq \varepsilon ^{3}\|a^{2}D^{4}u\|_{p}+\frac{C}{\varepsilon}\|u\|_{p} \label{eq:interp00-01}\\
\|aD^{2}u\|_{p}&\leq \varepsilon^2 \|a^{2}D^{4}u\|_{p}+\frac{C}{\varepsilon^2}\|u\|_{p}\label{eq:interp00-02} \\
\|a^{\frac{3}{2}}D^{3}u\|_{p}&\leq \varepsilon \|a^{2}D^{4}u\|_{p}+\frac{C}{\varepsilon^{3}}\|u\|_{p} \label{eq:interp00-03}\\
 \|a^2D^4u\|_p&\leq C\left( \|Lu\|_p+\|u\|_p\right) \label{eq:calderon-zigmund}
\end{align}
for every $u\in C_{c}^{\infty}(\R^{N})$.
\end{lemma}
\begin{proof}
By \cite[Lemma 2.4]{gre-spi-tac1} the inequalities \eqref{eq:interp00-01},\eqref{eq:interp00-02}, \eqref{eq:interp00-03}    and
\begin{equation} \label{CZ-A}
 \|a^2D^4u\|_p\leq C\left( \|Au\|_p+\|u\|_p\right)
\end{equation}
hold with $C=C(N,p,\kappa)$.
Then, we have
\begin{align*}
\|a^2D^4u\|_p\leq C\left( \|Lu\|_p+\|Bu\|_p+\|u\|_p\right).
\end{align*}
On the other hand, since
\[
-Bu=\sum_{i,j,k=1}^Nb_{ijk}D_{ijk}u+\sum_{i,j=1}^Nc_{ij}D_{ij}u+\sum_{i=1}^Nd_iD_iu+eu
\]
we have
\begin{align*}
\|Bu\|_p&\leq C\left( \sum_{i,j,k=1}^N\|b_{ijk}D_{ijk}u\|_p+\sum_{i,j=1}^N\|c_{ij}D_{ij}u\|_p+\sum_{i=1}^N\|d_iD_iu\|_p+\|eu\|_p \right) \\
&  \leq C\left(\kappa \|a^\frac{3}{2}D^3 u\|_p +\kappa^2\|a D^2u\|_p+\kappa^3\|a^\frac{1}{2}D u\|_p+\kappa^4\|u\|_p  \right).
\end{align*}
Then for every $\varepsilon>0$ we have
\[
\|Bu\|_p\leq \varepsilon \|a^2D^4u\|_p+C\kappa^4\|u\|_p.
\]
By (\ref{CZ-A}) and the last estimate,
$$\|a^2D^4u\|_p\leq C\left( \|Au\|_p+\|u\|_p\right)\leq C\left( \|Lu\|_p+\varepsilon \|a^2D^4u\|_p+\|u\|_p\right)$$
and the claimed inequality \eqref{eq:calderon-zigmund} follows by choosing $\varepsilon$ small enough.
\qed
\end{proof}

By the previous Lemma we  deduce that $B$ is a small perturbation of $A$ and then $L=A+B$ generate an analytic semigroup.

\begin{theorem}\label{th:gen}
Let  $1<p<\infty$ and assume that $a,b,c,d,e$ satisfy (\ref{gradient-p}). Then $(L,D(A_p))$ generates an analytic semigroup in $L^p(\R^N)$.
 Moreover, the analyticity sector depends on $N,p$ and $\kappa$. In  particular, there exists $r_p>0$ such that  the equation $\lambda u-Lu=f$ is uniquely solvable in $D(A_p)$ for $f\in L^p(\R^N)$ and  $\Rp\lambda>r_p$.
\end{theorem}
\begin{proof}
Let us endow
$B$   with its maximal domain
$$D(B)=\{ u \in W_{\rm loc}^{3,p}(\R^N)\cap L^p(\R^N )\,|\, Bu\in L^p(\R^N) \}.$$
We prove that $B$ is relatively $A$-bounded with $A$-bound $0$. First we observe that $D(A_p)\subset D(B)$. Indeed, let $u\in D(A_p)$, we have
\begin{align*}
&\|Bu\|_p \leq C\left(\|a^\frac{3}{2}D^3 u\|_p +\|a D^2u\|_p+\|a^\frac{1}{2}D u\|_p+\|u\|_p  \right) \leq C\|u\|_{D(A_p)}.
\end{align*}
Now, let $u\in D(A_p)$ and, by Lemma \ref{density}, let $u_n\in C_c^\infty(\R^N)$ be a sequence converging to $u$ with respect to the norm $\|\cdot \|_{D(A_p)}$.
By Lemma \ref{interpolation}, for every $\varepsilon>0$  there exists $C=C(N,p,\kappa)$ such that
\begin{align*}
\|Bu_n\|_p&\leq  C\left(\|a^\frac{3}{2}D^3 u_n\|_p +\|a D^2u_n\|_p+\|a^\frac{1}{2}D u_n\|_p+\|u_n\|_p  \right)\\
& \leq \varepsilon \|a^2D^4u_n\|+C\|u_n\|_p\leq \varepsilon C_1\|Au_n\|_p+C_2\|u_n\|_p.
\end{align*}
Taking the limit as $n\to\infty$ and taking into account that $Bu_n\to Bu$ and $Au_n \to Au$ we have that for every $\varepsilon>0$ there exists $C>0$ such that
\[
\|Bu\|_p\leq \varepsilon \|Au\|_p +C\|u\|_p.
\]
By \cite[Theorem III.2.10]{eng-nag}
 and since $C=C(N,p,\kappa)$
we have the desired result.
\qed
\end{proof}

As a consequence we obtain the following result. 

\begin{corollary}\label{agmon-comp-2}
Let  $1<p<\infty$ and assume that $a,b,c,d,e$ satisfy (\ref{gradient-p}). Then, there exist $r_p=r_p(N,\kappa)>0$ and $C=C(N,p,\kappa)$ such that  for every  $\Rp\lambda>r_p$ and $u\in D(A_p)$
\begin{equation*}
|\lambda|\|u\|_p+|\lambda|^\frac{3}{4}\|a^{\frac{1}{2}}Du\|_{p}+|\lambda|^\frac{1}{2}\|aD^{2}u\|_{p}+|\lambda|^\frac{1}{4}\|a^{\frac{3}{2}}D^{3}u\|_{p}+ \|a^2D^4u\|_p\leq C \|\lambda u-Lu\|_p.
\end{equation*}
\end{corollary}
\begin{proof}
By Theorem \ref{th:gen},  $L$ generates an analytic semigroup and there exists a sector $\Sigma$ depending on $N,p$ and $\kappa$ such that
\[
|\lambda|\|u\|_p\leq C\|\lambda u-Lu\|_p
\]
for some positive constant $C$ and for every $\lambda \in \Sigma$.
The thesis now follows by    Lemma \ref{interpolation}.
Indeed,   for $h=1,2,3$
\begin{align*}
&|\lambda|^\frac{4-h}{4}\|a^\frac{h}{2}D^hu\|_p\leq
|\lambda|^\frac{4-h}{4}\varepsilon^{4-h} \|a^2D^4u\|_p+|\lambda|^\frac{4-h}{4}\frac{C}{\varepsilon^h}\|u\|_p.
\end{align*}
Setting  $\varepsilon'=\varepsilon|\lambda|^\frac14$ one has $\frac{|\lambda|^\frac{4-h}{4}}{\varepsilon^h}=\frac{|
\lambda|}{\left( \varepsilon' \right)^h}$ and then
\begin{align*}
|\lambda|^\frac{4-h}{4}\|a^\frac{h}{2}D^hu\|_p&\leq
(\varepsilon ')^{4-h} \|a^2D^4u\|_p+\frac{C}{(\varepsilon')^h}|\lambda|\|u\|_p\\
&  \leq (\varepsilon ')^{4-h}(\|Lu\|_p+\|u\|_p)+C|\lambda|\|u\|_p\\
& \leq C\|\lambda u-Lu\|_p
\end{align*}
for $|\lambda|$ large enough. 
\qed
\end{proof}

At this point,  using  Corollary \ref{agmon-comp-2} and  the Masuda-Stewart localization technique as in \cite[Section 4]{gre-spi-tac1}, we can prove  generation results in  $C_b(\R^N)$ and $L^\infty(\R^N)$ for $L$.  As in \cite[Section 4]{gre-spi-tac1} for the operator $A$, we should assume more restrictive assumptions on $a$.  In particular, following the proof in the quoted paper,  it can be seen that we need  $a(x)\geq \nu >0$.

Let us define
\begin{align*}
&D(L_{0})\\&\ =\left\{ u\in \underset{1\leq p <\infty}{\bigcap} W^{4,p}_{\rm loc}(\R^{N})\cap C_b(\R^N)\,|\,
		a^{\frac h2}D^{h}u, a^2\Delta^2u\in C_{b}(\R^{N}), h=0,1,2, 3 \right\}.
\end{align*}

\begin{theorem}\label{generatcb}
Assume that $a,b,c,d,e$ satisfy (\ref{gradient-p}) and, in addition, $ a(x)\geq \nu >0$. Then $(L,D(L_0))$ generates an analytic semigroup in $C_{b}(\R^N)$.
 Moreover, the analiticity sector depends on $N,p$ and $\kappa$.
\end{theorem}

\section{Generation results in $L^1$}\label{L1}
For simplicity, here we consider operators containing only the higher order derivatives, that is  $A=-a(x)^2\Delta^2$.
We require the assumptions:
\begin{alignat*}{2} \label{gradient-1} a \in C^4(\R^N),&\  a(x)\geq \nu> 0\\
 |D_{ijkl}a(x)|&\leq \kappa ^4a^{-1}(x)\quad &&\forall\,i,j,k,l=1,\dots, N,\\\tag{$H_1$} |D_{ijk}a(x)|&\leq \kappa^3a^{-\frac{1}{2}}(x) &&\forall\,i,j,k=1,\dots, N,\\ |D_{ij}a(x)|&\leq \kappa^2 && \forall\,i,j=1,\dots, N,\\ |D_{i}a(x)|&\leq \kappa a^\frac{1}{2}(x) &&\forall\,i=1,\dots, N,\end{alignat*}
for all $x \in \R^N$ and some positive constants $\kappa$, $\nu$.
We will deal with the formal adjoint operator
$$
-\hat Au =a^2\Delta ^2u+4D a^2D \Delta u
+2\Delta a^2\Delta u+4tr D^2a^2D^2
u+4D \Delta a^2D u+\Delta ^2a^2 u.$$
Observe that conditions \eqref{gradient-1} on $A$ implies that  $\hat A$ satisfies the hypotheses of  Theorem \ref{generatcb}, whence
$\hat A$  endowed with the domain described before, generates an analytic semigroup  in $C_b(\R^N)$.

\begin{remark}
Observe that, if $0\leq\alpha\leq2$, the function $a(x)=(1+|x|^2)^{\frac{\alpha}{2}}$ satisfies (\ref{gradient-1}).
\end{remark}
We consider the operator $A=- a(x)^2 \Delta^{2}$ endowed with the domain
$$
D(A_1)
=\{u\in  L^1(\R^N):
\ a^\frac{1}{2}D u,\ a D^2u,\ a^\frac{3}{2}D^3 u,\ a^2\Delta ^2 u\in L^1(\R^N)\}.
$$
By \cite[Theorem 5.8 (i)]{tan}, $D(A_1)$ continuously embeds into $W^{3,p}(\R^N)$ for every $p\in \left[1,\frac{N}{N-1}\right[$ and, for every $u\in D(A_1)$,
\begin{equation*}\label{embT}
\|u\|_{W^{3,p}}\leq C_p(\|u\|_1+\|\Delta^2 u\|_1)\leq C_p \max\{1,\nu^{-2}\}(\|u\|_1+\|a^2\Delta^2 u\|_1).
\end{equation*}

\subsection{A-priori estimates and domain characterization}

We need  a density result, useful to work with smooth functions. The proof follows as in \cite[Lemma 2.2]{gre-spi-tac1}.

\begin{lemma}\label{density}
Assume that $a$ satisfies (\ref{gradient-1}). Then $C_c^\infty(\R^N)$ is dense in $D(A_1)$ with respect to the norm 
$$\|u\|_{D(A_1)}=\|u\|_1+\|a^\frac{1}{2}D u\|_1+\|a D^2u\|_1+\|a^\frac{3}{2}D^3u\|_1+\|a^2\Delta^2 u\|_1.$$
\end{lemma}

In order to prove some a-priori estimates, we need the following  covering result  (see \cite{cup-for} for its proof).
\begin{proposition}\cite[Proposition 6.1]{cup-for} \label{pr:covering}
Let ${\mathcal F} =\{  B_{\rho (x)}(x)\}_{x\in \R^{N}}$, 
where $\rho : \R^{N} \to \R_{+}$ is a Lipschitz continuous function with Lipschitz
constant $c$ strictly less than $\frac 12$.
Then, there exist a countable subcovering $\{ B_{\rho(x_{n})}(x_{n})\}_{n\in \N}$
 and a natural number $\xi=\xi (N,c)$ such that at most $\xi$
among the doubled balls $\{ B_{2\rho(x_{n})}(x_{n} )\}_{n\in \N}$ overlap.
\end{proposition}

Observe that by the assumption $|Da(x)|\leq\kappa a(x)^\frac12$, the function $\rho(x)=\frac{\eta}{2\kappa}a(x)^\frac{1}{2}$, $0<\eta\leq 1$, satisfies the assumptions of the previous covering Proposition with Lipschitz constant smaller than $\frac{\eta}{4}$.

We show some preliminary interpolative inequalities for the harmonic and bi-harmonic operators.
\begin{lemma}  \label{interp-bila}
 There exists $\overline{\eps}$ such that for every $0<\varepsilon<\overline{\eps}$, $u\in C_c^\infty(\R^N)$,
 
 \begin{equation}\label{int1}\|Du\|_1\leq\varepsilon^3\|\Delta u\|_1+\frac{C}{\eps}\| u\|_1 \end{equation}
 and 
 \begin{equation}	\label{int3}\|D^3 u\|_1\leq \varepsilon \|\Delta^2 u\|_1+\frac{C}{\varepsilon^{3}}\|u\|_1.\end{equation}
\end{lemma}
\begin{proof}
In the case $N=1$ the inequalities \eqref{int1} and \eqref{int3} follow by straightforward computations using Taylor's formula with integral remainder.

Let now $N>1$ and  $\lambda>\overline{\lambda}>0$. By setting $m=4$ in \cite[Theorem 5.7]{tan}, the fundamental solution $K_\lambda(x,y)$ of the operator $\lambda+\Delta^2$ satisfies
$$|D_x^3 K_\lambda (x,y)|\leq Ce^{-|\lambda|^\frac{1}{4}|x-y|}|x-y|^{1-N}.$$
It follows that for every $u\in C_c^\infty(\R^N)$
$$\|D^3 u\|_1\leq C|\lambda|^{-\frac14}\|\lambda u+\Delta^2u\|_1\leq C\left(|\lambda|^{\frac{3}{4}}\|u\|_1+|\lambda|^{-\frac14}\|\Delta^2 u\|_1\right).$$
Then, setting $|\lambda|=\eps^{-4}$, we get \eqref{int3}. 

The inequality \eqref{int1}  in the case $N>1$
follows   arguing as before setting $m=2$ in  \cite[Theorem 5.7]{tan}. 
\qed
\end{proof}

The next lemma constitutes a first step toward the a-priori estimates.  
\begin{lemma}  \label{interpolation-1}
Assume     $a$ satisfies (\ref{gradient-1}). Then, there exists a positive constant  $C=C(N,\kappa)$ such that for every $\varepsilon>0$ small enough the following inequalities hold
\begin{align}
&\|a^{\frac{h}{2}}D^hu\|_{1}\leq \varepsilon ^{4-h}\|Au\|_{1}+\frac{C}{\varepsilon^{h}}\|u\|_{1} \label{eq:interp00}
\end{align}
for $h=1,2,3$ and for every $u\in C_{c}^{\infty}(\R^{N})$. 
\end{lemma}
\begin{proof}

We first show that if $v\in C_{c}^{\infty}(\R^{N})$
\begin{equation}\label{eq:interp0}
\|a^{\frac{h}{2}}D v\|_{1}\leq \varepsilon \|a^{\frac{h+1}{2}}\Delta v\|_{1}+\frac{C}{\varepsilon}\|a^{\frac{h-1}{2}}v\|_{1}
\end{equation}
holds for every $h\in \N$.

Let us set $\rho(x)=\frac{1}{2\kappa}a(x)^\frac{1}{2}$ and note that, as stated before,  by \eqref{gradient-1} it follows that $\rho(x)$ satisfies the assumptions of Proposition \ref{pr:covering} with Lipschitz constant smaller than $\frac{1}{4}$. We fix  $x_{0}\in \R^{N}$ and set $R=\rho (x_{0})=\frac{1}{2 \kappa}a^{\frac{1}{2}}(x_{0})$. 
 Then, if $x\in B_{2R}(x_0)$ the following inequalities hold
\[
\frac{1}{2}a(x_{0})^{\frac{1}{2}}\leq a(x)^{\frac{1}{2}}\leq \frac{3}{2}a(x_{0})^{\frac{1}{2}}
\]
and consequently 
\[
\left (\frac{1}{2}\right)^{ k}a(x_{0})^{\frac{k}{2}}\leq a(x)^{\frac{k}{2}}\leq \left (\frac{3}{2}\right)^{k}a(x_{0})^{\frac{k}{2}}
\]
hold for every $k\in\N$. In particular, we will use the latter inequalities for $k=h-1,h,h+1$.

Let now  $\vartheta\in C_c^{\infty }(\R^N)$ be such that $0\leq \vartheta \leq 1$,
$\vartheta(x)=1$ for $x\in B_{R}(x_{0}),\vartheta(x)=0$ for $x\in B^{c}_{2R}(x_{0})$ and $|D^j\vartheta|\leq {C}{R^{-j}}$ for $j=1,2,3,4$.

By \eqref{int1} we have
\begin{align*}
\|a^{\frac{h}{2}}D v\|_{1, R}
&\leq C\|a(x_{0})^{\frac{h}{2}}D v\|_{1, R}\leq  C\|a(x_{0})^{\frac{h}{2}}D (\vartheta v)\|_{1}\\
& \leq \varepsilon  \|a(x_{0})^{\frac{h+1}{2}} \Delta (\vartheta v)\|_{1}
		+\frac{C}{\varepsilon } \| a(x_{0})^{\frac{h-1}{2}} \vartheta v\|_{1}\\
&  \leq \varepsilon  \|a(x_{0})^{\frac{h+1}{2}} \Delta v\|_{1,2R}+\frac{C}{R}\varepsilon \|a(x_{0})^{\frac{h+1}{2}} D v\|_{1,2R}
		+\frac{C}{R^{2}}\varepsilon \|a(x_{0})^{\frac{h+1}{2}} v\|_{1,2R}\\
&\quad		+\frac{C}{\varepsilon}\| a(x_{0})^{\frac{h-1}{2}} v\|_{1,2R}.
\end{align*}

Then,
taking into account that 
$\frac{1}{R}a(x_{0})^{\frac{h+1}{2}}=2\kappa a(x_{0})^{\frac{h}{2}}\leq Ca(x)^{\frac{h}{2}}$
and \\$\frac{1}{R^{2}}a(x_{0})^{\frac{h+1}{2}}=4\kappa^2 a(x_{0})^{\frac{h-1}{2}}\leq Ca(x)^{\frac{h-1}{2}}$ in $B_{2R}(x_0)$,
one has
\begin{align}\label{eq:cover-x00}
 \|a^{\frac{h}{2}}D v\|_{1,R} \leq C \left( \varepsilon \|a^{\frac{h+1}{2}}\Delta v\|_{1,2R}+\varepsilon \|a^{\frac{h}{2}}D v\|_{1,2R}
		+\left(\varepsilon+\frac{1}{\varepsilon}\right)\|a^{\frac{h-1}{2}}v\|_{1,2R}
		\right).
\end{align}

Let $\{B_{\rho(x_n )}(x_n)\}$ be a countable covering of $\R^N$ as in Proposition \ref{pr:covering}
such that at most $\zeta $ among the double balls $\{B_{2\rho(x_n )}(x_n)\}$ overlap.

We write \eqref{eq:cover-x00} with $x_0$ replaced by $x_n$ and sum over $n$.
Taking into account the above covering result, we get
\begin{align*}
\|a^{\frac{h}{2}}D v\|_{1}&\leq  \sum_{n\in\N}\|a^{\frac{h}{2}}D v\|_{1,R}
\\ &\leq 
C \sum_{n\in\N} \left( \varepsilon \|a^{\frac{h+1}{2}}\Delta v\|_{1,2R}+\varepsilon \|a^{\frac{h}{2}}D v\|_{1,2R}
		+\left(1+\frac{1}{\varepsilon}\right)\|a^{\frac{h-1}{2}}v\|_{1,2R}\nonumber
		\right)
\\ &\leq 
C\xi \left( \varepsilon \|a^{\frac{h+1}{2}}\Delta v\|_{1}+\varepsilon \|a^{\frac{h}{2}}D v\|_{1}
		+\left(1+\frac{1}{\varepsilon}\right)\|a^{\frac{h-1}{2}}v\|_{1}\nonumber
		\right).
\end{align*}
Choosing $\varepsilon'=\dfrac{C\xi\varepsilon}{1-C\xi\varepsilon}$   one can find a suitable positive $C$ such that
$$
\|a^{\frac{h}{2}}D v\|_{1}\leq 
\varepsilon' \|a^{\frac{h+1}{2}}\Delta v\|_{1}
+\dfrac{C}{\varepsilon'}\|a^{\frac{h-1}{2}}v\|_{1}\nonumber
$$
from which 
\eqref{eq:interp0} follows.

Now for every $i,j=1,\dots N$ we have 
\begin{align*}
\|aD_{ij}u\|_1&\leq \varepsilon_1 \|a^{\frac{3}{2}}\Delta D_ju\|_1+\frac{C}{\varepsilon_1}\|a^\frac{1}{2}D_ju\|_1\\
&  \leq \varepsilon_1\varepsilon_2 \|a^2\Delta^2u\|_1+C\frac{\varepsilon_1}{\varepsilon_2}\|a\Delta u\|_1
	+C\frac{\varepsilon_3}{\varepsilon_1}\|a\Delta u\|_1+C^2\frac{1}{\varepsilon_1\varepsilon_3}\|u\|_1\\
&  \leq 	\varepsilon_1\varepsilon_2 \|a^2\Delta^2u\|_1+C\left(\frac{\varepsilon_1}{\varepsilon_2}+\frac{\varepsilon_3}{\varepsilon_1}\right)
\|aD^2u\|_1+C\frac{1}{\varepsilon_1\varepsilon_3}\|u\|_1
\end{align*}
and then 
\[\|aD^2u\|_1\leq C\left(
\varepsilon_1\varepsilon_2 \|a^2\Delta^2u\|_1+\left(\frac{\varepsilon_1}{\varepsilon_2}+\frac{\varepsilon_3}{\varepsilon_1}\right)
\|aD^2u\|_1+\frac{1}{\varepsilon_1\varepsilon_3}\|u\|_1\right).
\]
Setting $\varepsilon_2=4C\varepsilon_1$, $\varepsilon_3=\frac{\varepsilon_1}{4C}$ and $\varepsilon ^2=8C^2\varepsilon_1^2$ we obtain
\eqref{eq:interp00} when $h=2$.

Then, we deduce
\begin{align*}
\|a^\frac{1}{2}D u\|_1\leq \varepsilon_1 \|a\Delta u\|_1+\frac{C}{\varepsilon_1}\|u\|_1
\leq \varepsilon_1\varepsilon_2^2\|Au\|_1+C\frac{\varepsilon_1}{\varepsilon_2^2}\|u\|_1+C\frac{1}{\varepsilon_1}\|u\|_1.
\end{align*}
Setting $\varepsilon_1=\varepsilon_2$ we obtain \eqref{eq:interp00} when $h=1$. 

Arguing as done before in order to prove  \eqref{eq:interp0}, applying inequality \eqref{int3}, one gets
\begin{equation}
\|a^{\frac{3}{2}}D^3 u\|_{1}\leq 
\varepsilon  \|A u\|_{1}
+\dfrac{C}{\varepsilon^3 }\|u\|_{1}\nonumber.
\tag*{\qed}
\end{equation}
\end{proof}

The following proposition is an immediate  consequence of the previous lemma and gives the announced a priori estimates.

\begin{proposition} \label{domain}
Assume   $a$ satisfies (\ref{gradient-1}). Then, for $u\in D(A_1)$ the following inequality holds
$$\|a^\frac{1}{2}D u\|_1+\|a D^2u\|_1+\|a^\frac{3}{2}D ^3u\|_1\leq C(\|Au\|_1+\|u\|_1)$$
where the  constant $C$ depends on $N$ and $\kappa$. 
\end{proposition}

We are now able to characterise the maximal domain of $A$ in $L^1(\R^N)$ arguing as in \cite[Proposition 2.8]{gre-spi-tac1}.

\begin{proposition} \label{domain2}
Assume $a$ satisfies (\ref{gradient-1}). Then 
$$
D(A_{1})=D(A_{1,\rm max})=\{u\in L^1(\R^N):\   Au\in L^1(\R^N)\}.$$
\end{proposition}

\subsection{Sectoriality and solvability}

To investigate the generation of an analytic semigroup in $L^1(\R^N)$, we start by proving an a-priori sectoriality estimate for the resolvent in a suitable sector of the complex plane. Then we prove the surjectivity.

\begin{theorem} \label{sectorial}
Assume that $a$ satisfies (\ref{gradient-1}). Then, there exists $\Sigma\subseteq\C$ such that, if $\lambda\in\Sigma$, the following inequality holds for every $u\in D(A_1)$
$$|\lambda|\|u\|_1\leq C\|\lambda u-Au\|_1$$
where $C=C(N,\kappa)$ and the sector $\Sigma$ depends on $\kappa$.
 \end{theorem}
\begin{proof}
We use a duality argument to get the sectoriality estimate in $L^1$.
Consider the formal adjoint
$$-\hat Au =a^2\Delta ^2u+4D a^2D \Delta u+2\Delta a^2\Delta u+4tr D^2a^2D^2
u+4D \Delta a^2D u+\Delta ^2a^2 u.$$
By the conditions \eqref{gradient-1} on $A$, the operator $\hat A$ generates an analytic semigroup in $C_b(\R^N)$  (Theorem \ref{generatcb}). Therefore there exists a sector $\Sigma$ contained in the resolvent set of $\hat{A}$ and
$$|\lambda|\|u\|_\infty\leq C\|\lambda u-\hat{A}u\|_\infty$$ for every $\lambda\in\Sigma$, $u\in C_b(\R^N)$ and for some constant $C$ independent of $\lambda$ and $u$. In view of Lemma \ref{density}, we prove the sectoriality estimate for functions in $C_c^\infty(\R^N)$.
Let $\lambda\in\Sigma$, $u\in C_c^\infty(\R^N)$. Then, since $(\lambda I-\hat{A})(D(\hat{A}))=C_b(\R^N)$, we have
\begin{align*}
\|u\|_1&\leq \sup\left\{\int_{\R^N}u(\lambda I-\hat{A})v\, dx:\ v\in D(\hat{A}),\ \|\lambda v-\hat{A}v\|_\infty\leq 1\right\}\\
       &\leq \sup\left\{\int_{\R^N}u(\lambda I-\hat{A})v\, dx:\ v\in D(\hat{A}),\ \|v\|_\infty\leq \frac{C}{|\lambda|}\right\}.
\end{align*}
Integrating by parts, we get
\begin{align}
\left|\int_{\R^N}u(\lambda I-\hat{A})v\right|= \left|\int_{\R^N}v(\lambda I-A)u\right|\leq \frac{C}{|\lambda|}\|\lambda u-Au\|_1.
\tag*{\qed}
\end{align}
\end{proof}

The previous estimate leads, through an approximation procedure, to the solvability of the elliptic problem. 

\begin{theorem}\label{solv}
Assume that $a$ satisfies (\ref{gradient-1}). Then there exists $\Sigma\subseteq\C$ such that  the equation $\lambda u-Au=f$ is uniquely solvable in $D(A_1)$ for $f\in L^1(\R^N)$ and  $\lambda\in\Sigma$  and the resolvent estimate 
$$|\lambda|\|u\|_1\leq C\|\lambda u-Au\|_1$$ holds.
\end{theorem}
\begin{proof} 
Let $\sigma>0$ and define the function
$$a_\sigma(x)=\frac{a(x)}{1+\sigma a(x)},\quad x\in\R^N.$$ Observe that $a_\sigma(x)$ satisfies (\ref{gradient-1}) with  $\kappa$ independent of $\sigma$. Consider the operators with bounded coefficients $A_\sigma=-(a_\sigma(x))^2\Delta^2$. 
By \cite[Section 5.4]{tan}, there exist $\omega_{\sigma}\in\R$, $M_{\sigma}>0$ such that $\{\lambda\in\C:\ \Rp \lambda\geq \omega_{\sigma}\}\subset\rho(A_{\sigma})$ and 
\begin{equation*} 
|\lambda|\|u\|_1\leq M_{\sigma}\|\lambda u-A_\sigma u\|_1
\end{equation*} for every $u\in D(A_{\sigma,1,\rm max})$ and $\lambda\in\C$ with $\Rp \lambda\geq \omega_{\sigma}$.
Moreover, by applying Theorem  \ref{sectorial}  to $A_\sigma$, since the condition in (\ref{gradient-1}) is satisfied with a constant independent of $\sigma$, we have that there exists $r_1>0$  (independent of $\sigma$) such that  for every  $\lambda\in\C$ with $\Rp\lambda>r_1$ and for every $u\in D(A_{\sigma,1,\rm max})$
\begin{equation} \label{stima-ind}
|\lambda|\|u\|_1\leq C_1\|\lambda u-A_\sigma u\|_1
\end{equation}
for some constant $C_1$ independent of $\sigma$.  Let $\overline{\lambda}\in \rho (A_{\sigma })$ and $\Rp\overline{\lambda}>r_1$, then the inequality (\ref{stima-ind}) gives that
\begin{equation} \label{resolvent}
\|R(\overline{\lambda},A_\sigma)\|_1\leq \frac{C_1}{|\overline{\lambda}|}\leq \frac{C_1}{r_1}.
\end{equation}
By classical result, if $|\lambda-\overline{\lambda}|<\frac{1}{\|R(\overline{\lambda},A_\sigma)\|_1}$, then  $\lambda\in\rho (A_{\sigma })$. By (\ref{resolvent}), if
$|\lambda-\overline{\lambda}|<\frac{r_1}{C_1}$, then  $\lambda\in\rho (A_{\sigma })$.  
By covering with balls of radius $\frac{r_1}{C_1}$, it follows that, if $\Rp\lambda>r_1$, then $\lambda\in\rho(A_{\sigma })$.
In particular, $\lambda=\overline{\lambda}-\frac{\overline{\lambda}}{2C_1}$ satisfies $|\lambda-\overline{\lambda}|<\frac{|\overline{\lambda}|}{C_1}$ and then  belongs to $\rho (A_\sigma)$.

Let $\Rp\lambda>r_1$, $f\in L^1(\R^N)$. Denote by $u_\sigma\in D(A_{\sigma,1,\rm max})$ the unique solution of $\lambda u_\sigma-A_\sigma u_\sigma=f$. Then 
by \cite[Theorem 5.8 (i)]{tan}, $D(A_{\sigma,1,\rm max})$ continuously embeds into $W^{3,p}(\R^N)$ for every $p\in \left[1,\frac{N}{N-1}\right[$ and, observing that $a_\sigma\geq\frac{\nu}{\nu+1}$, for every $u\in D(A_{\sigma,1,\rm max})$, 
\begin{equation}\label{emb}
\|u_\sigma\|_{W^{3,p}}\leq C_p(\|u_\sigma\|_1+\|\Delta^2 u_\sigma\|_1)\leq C_p(\nu)(\|u_\sigma\|_1+\|a_\sigma^2\Delta^2 u_\sigma\|_1)\leq  C_p(\nu,r_1)\|f\|_1.
\end{equation}
From (\ref{emb}) we deduce that, up to considering a subsequence, we can assume that $u_\sigma$ converges weakly in $W^{3,p}(\R^N)$ to some function $u\in W^{3,p}(\R^N)$ as $\sigma$ goes to $0$ and strongly in $L^p_{\rm loc}(\R^N)$.
In particular, $u_\sigma$ converges to $u$ in $L^1_{\rm loc}(\R^N)$. From (\ref{stima-ind}), we deduce that $u\in L^1(\R^N)$ and
 \begin{equation*} 
|\lambda|\|u\|_1\leq C_1\|f\|_1.
\end{equation*}
It remains to show that u belongs to $D(A_1)$ and solves the equation $\lambda u-Au=f$. Fix $v\in C_c^4(\R^N)$ and observe that
\begin{equation}  \label{int-parts}
\int_{\R^N}a_\sigma^2\Delta^2u_\sigma\,v\, dx=\int_{\R^N}u_\sigma\Delta^2(a_\sigma^2 v)\, dx.
\end{equation}
By easy computations, 
$\Delta^2(a_\sigma^2 v)$ converges uniformly in $\R^N$ to $\Delta^2(a^2 v)$.
Letting $\sigma$ to $0$ in $\lambda u_\sigma-A_\sigma u_\sigma=f$, we get $A_\sigma u_\sigma\to \lambda u-f$ and in the equation (\ref{int-parts}) we obtain
\begin{equation*}  
\int_{\R^N}\frac{\lambda u-f}{a^2}\,a^2v\, dx=\int_{\R^N}u\Delta^2(a^2 v)\, dx.
\end{equation*}
It follows that
\begin{equation*}  
\int_{\R^N}\frac{\lambda u-f}{a^2}\,w\, dx=\int_{\R^N}u\Delta^2 w\, dx
\end{equation*}
for every $w\in C_c^4(\R^N)$. This means that $a^2\Delta^2 u\in L^1(\R^N)$ and $u$ satisfies $\Delta^2 u=\frac{\lambda u-f}{a^2}$ in the distributional sense.
\qed
\end{proof}

In view of \cite[Proposition 3.2.8]{lor-rha-book} the operator $A$ is sectorial. Therefore, by standard generation results, cf. \cite[Ch. II, Theorem 4.6]{eng-nag}, we get the following generation theorem.
\begin{theorem}
Assume that $a$ satisfies (\ref{gradient-1}). Then, $(A,D(A_1))$ generates an analytic semigroup in $L^1(\R^N)$.
\end{theorem}

\section*{Acknowledgement}
 The authors are members of the Gruppo Nazionale per l’Analisi Matematica, la Probabilità e le loro Applicazioni (GNAMPA) of the Istituto Nazionale di Alta Matematica (INdAM). This article is based upon work from COST Action CA18232 MAT-DYN-NET, supported by COST (European Cooperation in Science and Technology), www.cost.eu.

\bibliographystyle{amsplain}

\bibliography{bibfile}

\end{document}